\documentclass [a4paper]{article}

\usepackage{psfrag}
\usepackage[dvips]{graphicx}
\usepackage{subfigure}
\usepackage{amsmath}
\usepackage{amsthm}
\usepackage{amssymb}

\newtheorem{theorem}{Theorem}[section]

\newtheorem{lemma}[theorem]{Lemma}

\theoremstyle{definition}
\newtheorem{definition}{Definition}

\theoremstyle{remark}

\usepackage{kbordermatrix}

\makeatletter
\def\bbordermatrix#1{\begingroup \m@th
  \@tempdima 4.75\p@
  \setbox\z@\vbox{%
    \def\cr{\crcr\noalign{\kern2\p@\global\let\cr\endline}}%
    \ialign{$##$\hfil\kern2\p@\kern\@tempdima&\thinspace\hfil$##$\hfil
      &&\quad\hfil$##$\hfil\crcr
      \omit\strut\hfil\crcr\noalign{\kern-\baselineskip}%
      #1\crcr\omit\strut\cr}}%
  \setbox\tw@\vbox{\unvcopy\z@\global\setbox\@ne\lastbox}%
  \setbox\tw@\hbox{\unhbox\@ne\unskip\global\setbox\@ne\lastbox}%
  \setbox\tw@\hbox{$\kern\wd\@ne\kern-\@tempdima\left[\kern-\wd\@ne
    \global\setbox\@ne\vbox{\box\@ne\kern2\p@}%
    \vcenter{\kern-\ht\@ne\unvbox\z@\kern-\baselineskip}\,\right]$}%
  \null\;\vbox{\kern\ht\@ne\box\tw@}\endgroup}
\makeatother

\newfont{\valami}{ptmr8r scaled 1200} 
\newfont{\kisvalami}{ptmr8r scaled 1000}

\begin{document}

\title{$k$-Sum Decomposition of Strongly Unimodular Matrices}

\author{K. Papalamprou and  L. Pitsoulis\thanks{work of this author was conducted at National Research University Higher School of Economics and supported by RSF grant 14-41-00039} \\
                Department of Electrical and Computer Engineering \\ 
                     Aristotle University of Thessaloniki, Greece \\
                   \texttt{papalamprou@auth.gr, pitsouli@auth.gr} }

\maketitle

\begin{abstract}
Networks are frequently studied algebraically through matrices. In this work, we show that networks may be studied in a more abstract level using results from the theory of matroids by establishing connections to networks by  decomposition results of matroids. First, we present the implications
of the decomposition of regular matroids to networks and related classes of matrices, and secondly  we show that  strongly unimodular matrices are  closed under $k$-sums for $k=1,2$ implying a decomposition into highly connected network-representing blocks,  which are also shown to have a special structure.    
\end{abstract}


\section{Introduction}
It is widely accepted that networks play an important role in many aspects of today's life. To name a few, social networks play an important role in relationships, job hunting, and marketing, while economic networks usually determine the sustainability and development of various organisations.
Moreover,  the understanding of complex biological networks may be the key in answering important questions in the areas of medicine and biology. For many other applications  as well as for an extended overview of the approaches related to complex networks, the interested reader is referred to~\cite{EasKle:10, Jack:10,New:2010}

Networks are naturally modelled as graphs and results from graph theory have been employed to explore and attack problems in networks (see e.g.~\cite{Steen:10}). Graphs are known to be represented algebraically via matrices and it is such a representation that has been extensively used in various problems concerning networks. In this work, we examine special classes of matrices that are related to networks and furthermore have important implications in optimisation. Our primary result shows that these matrices and, therefore the related networks, are decomposed into highly connected blocks which represent networks with specific properties. To do so we employ results from matroid decomposition theory and, to the best of our knowledge, this is among the few works using such tools to study complex networks. The purpose of this work is twofold; at the one hand we would like to relate complex networks with optimisation problems via well-known classes of matrices and on the other hand to present decomposition results for such classes of matrices and discuss their implication to networks. From our viewpoint, the main implication is the possibility of finding a way to study complex networks via exploring the properties of the building blocks that arise from specific decompositions.

The organization of the paper is as follows. 
In Section~\ref{sec_pre}, we provide the relevant theory and some preliminary results regarding matrices and  matroids in order to 
make this work more self-contained.
In Section~\ref{sec_ksm}, we focus on strongly unimodular matrices and show that they are closed under the $k$-sum operations $(k=1,2)$ and, based on that, how these matrices can be decomposed into smaller strongly unimodular matrices. The special structure of these smaller matrices is discussed in 
Section~\ref{sec_3c}. In the last section, the final decomposition result is provided along with the description of the associated highly connected building blocks.

\section{Special Matrices and Matroids} \label{sec_pre}
\subsection{Network and Unimodular Matrices}
We assume that the reader is familiar with the basic notions of graph theory as they are presented in~\cite{Diestel:05}.
Totally unimodular (TU) matrices form an important class of matrices for integer and linear programming due to the integrality properties of 
the associated polyhedron. A matrix $A$ is \emph{totally unimodular} if each square submatrix of $A$ has determinant $0,+1,$ or $-1$. The class of TU matrices has been studied extensively and combinatorial characterisations for these matrices can be found  in~\cite{NemWols:1988,Schrijver:86}. An important subclass of TU matrices is defined as follows.
A matrix $A$ is \emph{strongly unimodular} (SU) if: (i) $A$ is TU, and (ii) every matrix obtained from $A$ setting a $\pm{1}$ entry to $0$ is also TU. Another well-known characterisation for SU matrices goes as follows:
a matrix is strongly unimodular if any of its nonsingular submatrices is triangular,  where a triangular matrix is a square matrix whose entries below or above the main diagonal can become zero by permutation of rows or column. 
Strongly unimodular matrices have appeared several times in the literature \cite{Cora:87,CraLoPo:92,LoPo:89} since they were first introduced in \cite{CraHaIb:86}. Another  subclass of TU matrices discussed in this paper is the class of network matrices. A \emph{network matrix} may be viewed as an edge-path matrix of a directed graph with respect to a particular spanning tree of the graph; results regarding network matrices can be found in~\cite{NemWols:1988,Schrijver:86}. Seymour has shown in~\cite{Seymour:1980} that network matrices and their transposes are the main building blocks for TU matrices.  Moreover, in~\cite{PiPa:09}, it has been shown that the building blocks of TU matrices are matrices associated with bidirected graphs. In this paper we focus on SU matrices which stand between the classes of network and TU matrices and show the network structure of that class.

\subsection{Matroid Theory}
The main reference for matroid theory is the book of Oxley~\cite{Oxley:06}. 
\begin{definition} \label{def_ntef}
A matroid $M$ is an ordered pair $(E,\mathcal{I})$ of a finite set $E$  and a collection $\mathcal{I}$ of subsets of $E$ satisfying the following three conditions:
\begin{itemize}
\item [(I1)] $\emptyset \in{\mathcal{I}}$
\item[(I2)] If $X\in{\mathcal{I}}$ and $Y\subseteq{X}$ then $Y\in{\mathcal{I}}$ 
\item[(I3)] If $U$ and $V$ are members of $\mathcal{I}$ with $|U|<|V|$ then there exists $x\in{V-U}$ such that $U\cup{x}\in{\mathcal{I}}$. 
\end{itemize}
\end{definition}    
Given a matroid $M=(E,\mathcal{I})$, the set $E$ is called the \emph{ground set} of $M$ and the members of $\mathcal{I}$ are the \emph{independent sets} of $M$. Furthermore, any subset of $E$ not in $\mathcal{I}$ is called a $\emph{dependent set}$ of $M$ while a minimal dependent set is called a \emph{circuit} of $M$.

Let $E$ be a finite set of vectors from a vector space over a field $\mathbb{F}$ and let $\mathcal{I}$ be the collection of linearly independent subsets of $E$; then it can be proved that $M=(E,\mathcal{I})$ is a matroid called \emph{vector matroid} denoted by $M[A]$ where $A$ is a matrix  whose columns are the vectors of the ground set. It can be easily shown that there is one-to-one correspondence between the linearly independent columns of $A$ and the independent sets of $M$, so the matroid $M$ can be fully characterised by matrix $A$. Matrix $A$ is called a \emph{representation matrix} of $M$ and we also say that $M$ is $\mathbb{F}$-representable where $\mathbb{F}$ is the field that the elements of matrix $A$ belong. Suppose now that we delete from $A$ all the linearly dependent rows and from the matrix $A'$ so-obtained we choose a basis $B$. Clearly, linear $\mathbb{F}$-independence of columns is not affected by such a deletion of rows. By pivoting on non-zero elements of $B$ we can transform $A'$ to  matrix $[I\;B']$. Pivoting does not affect linear $\mathbb{F}$-independence of a matrix and, thus, $M=M[I\;B']$. The matrix $B'$  is called a \emph{compact representation matrix} of $M$. Two matrices are \emph{projectively equivalent} if one can be obtained from the other by elementary row operations and nonzero column scaling. A matroid $M$ is called \emph{uniquely representable} over some field $\mathbb{F}$ if and only if any two representation matrices of $M$ (over $\mathbb{F}$) are projectively equivalen. A matroid representable over every field is \emph{regular}. Furthermore, there is a clear connection between regular matroids and TU matrices. Specifically, any TU matrix is the representation matrix of some regular matroid and any regular matroid has a TU representation matrix (in $\mathbb{R}$).

Let $G$ be an ordinary graph and let $\mathcal{I}$ be the collection of edge sets inducing a acyclic subgraph of $G$. Then it can be shown that the pair $(E(G), \mathcal{I})$ is a matroid called the \emph{graphic matroid} of $G$  and is denoted by $M(G)$. If $A$ is the incidence matrix of an orientation of $G$ (i.e. the directed graph obtained from $G$ by assigning a direction to each edge)  then it can be shown that $M(G)$ is isomorphic to $M[A]$ and we write $M(G)\cong{M[A]}$. Thus, for any network matrix $N$ with respect to some 
spanning tree of $G$ we have that $M(G)\cong{M(N)}$, since the way we obtain  $N$ from $A$ is also the way we can obtain from $A$ a compact representation matrix of $M[A]$.
The ordered pair $(E,\{E-{S}:S \notin{\mathcal{I}}\})$ is a matroid called the \emph{dual matroid} of $M$ and is denoted by $M^{*}$. It is clear that $(M^{*})^{*}=M$. The prefix 'co' is used to dualize a term; therefore, a matroid is called cographic if it is the dual of a graphic matroid. We should note that not all matroids are closed under duality; for example  regular matroids are closed while graphic matroids are not. 

Any matroid which can be obtained from $M$ by a series of operations called {\em deletions} and {\em contractions} is called a \emph{minor} of $M$ (see e.g. Section~3.1 in~\cite{Oxley:06}). The \emph{rank} of  a matroid $M$, denoted by $r(M)$, equals the cardinality of the maximal independent set of $M$. For some positive integer $k$, a partition $(X,Y)$ of $E(M)$ is called a \emph{$k$-separation} of $M$ if the following two conditions are satisfied: (i) $\min\{|X|,|Y|\}\geq k$, and (ii) $r_{M}(X)+r_{M}(Y)-r(M) \leq k-1$.
Finally, we say that $M$ is \emph{$k$-connected} when it does not have an $l$-separation for $1\leq l \leq k-1$. 



\section{A $k$-sum Decomposition of Strongly Unimodular Matrices} \label{sec_ksm}

The following two results (Lemmas~\ref{lem_ew} and~\ref{lem_su22}) can be obtained easily from the definition of SU matrices and the fact that TU matrices are closed under deletions of rows and columns~\cite{NemWols:1988}. The proof of Lemma~\ref{lem_ew} is straightforward and is ommited.
\begin{lemma} \label{lem_ew}
Every submatrix of a strongly unimodular matrix is strongly unimodular.
\end{lemma} 

\begin{lemma} \label{lem_su22}
A TU matrix having at most two non-zeros in every column (row) is SU.
\end{lemma}
\begin{proof}
Let $A$ be a TU matrix with at most two non-zeros in every column. The case in
which $A$ has two non-zeros in every row can be handled in much the same way.  Let us set
a nonzero of column $i$ of $A$ to $0$ and call $A'$ the matrix so-obtained. Now every submatrix of $A'$ either is equal to the corresponding
submatrix of $A$; or we can expand the determinant of the submatrix of $A'$
along column $i$ (which has at most one nonzero being $\pm{1}$) and observe that the
determinant of $A'$ is actually equal, up to $\pm{1}$ scaling, to
the determinant of a submatrix of $A$.
\end{proof}

\noindent
As shown in the following result, SU matrices are closed  under fundamental matrix operations.  
\begin{lemma}  \label{lem_opra}
SU matrices are closed under the following operations:
\begin{itemize}
\item [(i)] transposing,
\item[(ii)] adding a zero row or column,
\item [(iii)] adding a unit column or a unit row, and
\item [(iv)] repeating a column or a row
\end{itemize}
\end{lemma}
\begin{proof}
Part (i) is trivial since the determinant of any submatrix remains unchanged under transposing. For (ii), let $A'$ be the matrix obtained from the addition of a zero row or column to a matrix $A$.Clearly, the replacement of any nonzero of $A'$ by a zero has to take place to the submatrix of $A'$ which is equal to $A$. But $A$ is SU and therefore we have that the matrix so-obtained is a TU matrix plus a zero column (row). The result now follows from the fact that TU matrices are closed under the addition of a zero row or column~\cite{Schrijver:86}.  

For (iii), let's add a unit column $a$ to an SU matrix $A$ and let's call $A'=[A\;a]$ the matrix so-obtained. The case in which a unit row is added can be handled similarly. If we change the nonzero of column $a$ to zero then this is equivalent of adding a zero row to a TU matrix and therefore the matrix so-obtained remains TU. If we change any other nonzero of $A'$ to zero then this has to be an element of the part $A$ of $A'$; let us change such a nonzero to zero and call $A''=[B\;a]$ the new matrix. We shall show that any submatrix of $A''$ is TU. Obviously, any submatrix of $B$ is TU because $A$ is an SU matrix. In the remaining case, we can expand the determinant of a submatrix along column $a$ and observe that this determinant is a $\pm{1}$ multiple of the determinant of a submatrix of $B$.       

For (iv), let $A'=[A\; a_1]$ be an SU matrix and let $a_1$ be a column of $A'$
which we repeat in order to construct the matrix $[A \;a_1 \;a_1]$. We note
here that the case of repeating a row can be handled in the same way. The only
case which has to be examined is the one in which a nonzero element of a
column $a_1$ becomes zero, since for all the other cases all the submatrices
of the matrix obtained are easily checked to be TU. Let $a_1'$ be the matrix
obtained from turning a nonzero of a column $a_1$ to zero, then the only
submatrices of $A''=[A \; a_1' \; a_1]$ which has to be examined of being TU are
those containing parts of column $a_1'$ and $a_1$, since all the other
submatrices are trivially TU. After expanding now the determinant of such a
submatrix of  $A''$ along the column $a_1'$ and also expanding the determinant of
the same submatrix of $[A \; a_1 \; a_1]$ along $a_1$ we see that these two
determinants  differ by a  determinant of a TU matrix. Thus, these
determinants differ by $0$ or $\pm{1}$. But the determinant of the submatrix
of  $[A \; a_1 \; a_1]$ is equal to zero and therefore we have that the
determinant of the corresponding submatrix of $A''$  is either $0$ or $\pm{1}$.            
\end{proof}

In what follows the operations of $k$-sum $(k=1,2,3)$ are of central importance.  
\begin{definition}\label{def_k-sums}
If $A, B$ are matrices, $a,d$ are column vectors and $b,c$ are row vectors of
appropriate size in $\mathbb R$ then we define the following matrix operations
\begin{description}
\item[1-sum:] $A\oplus_1 B:=\begin{bmatrix}A & 0 \\ 0 & B\end{bmatrix}$
\item[2-sum:] $\begin{bmatrix}A & a\end{bmatrix}\oplus_2 \begin{bmatrix}b\\B\end{bmatrix}:=\begin{bmatrix}A & ab\\0& 
B\end{bmatrix}$
\item[3-sum:] $\begin{bmatrix}A & a & a\\c & 0 & 1\end{bmatrix}\oplus_3 \begin{bmatrix}1 & 0 & b\\d & d & B\end{bmatrix}:=
\begin{bmatrix}A & ab\\dc& B\end{bmatrix}$ or \\
\hspace*{.13in}$\begin{bmatrix}A & 0 \\b & 1 \\ c & 1 \end{bmatrix}\oplus^3 \begin{bmatrix}1 & 1 & 0\\
 a& d & B\end{bmatrix}:=\begin{bmatrix}A & 0 \\ D & B\end{bmatrix}$ \\
where, in the $\oplus^3$, $b$ and $c$ are $\mathbb{R}$-independent row vectors and $a$ and $d$ are $\mathbb{R}$-independent column vectors
such that $[\frac{b}{c}]=[D_1|\bar{D}]$, $[a|d]=[\frac{\bar{D}}{D_2}]$ and $\bar{D}$ is a square non-singular matrix. 
Then, $D=[a|d]\bar{D}^{-1}[\frac{b}{c}]$.
\end{description}
\end{definition}

Network matrices and TU matrices are known to be closed under these operations (see \cite{PiPa:09} and \cite{Schrijver:86}, respectively). These operations have been originally defined in the more general framework of regular matroids in~\cite{Seymour:1980} and here we present the special form of these operations as they applied to the compact representation matrices in $\mathbb{R}$, i.e. TU matrices.

In the lemmas that follow we show that SU matrices are  closed under the $1$-sum and $2$-sum operations.
\begin{lemma} \label{lem_1-s}
If $A$ and $B$ are SU matrices then the matrix
$N=A\oplus_{1}B=\begin{bmatrix}A & 0\\0& B\end{bmatrix}$ is an SU matrix.
\end{lemma}
\begin{proof}
Since $A$ and $B$ are TU and from the fact that TU matrices are closed under $1$-sums we have that $N$ is TU. It remains to be shown that if we change a nonzero of the submatrix $A$ (or $B$) of $N$ to zero then the matrix $N'=\begin{bmatrix}A' & 0\\0& B'\end{bmatrix}$ obtained by this change is TU. Since $A$ and $B$ are SU we have that $A'$ and $B'$ are TU and from the fact that TU matrices are closed under the $1$-sum operation we have that $N'$ is TU as well. 
\end{proof}

\begin{lemma} \label{lem_2-s}
If $A=\begin{bmatrix}A' & a\end{bmatrix}$ and
$B=\begin{bmatrix}b\\B'\end{bmatrix}$ are SU matrices then the matrix
$N=A\oplus_2 B=\begin{bmatrix}A' & ab\\0& 
B'\end{bmatrix}$ is an SU matrix.
\end{lemma}
\begin{proof}
Since TU matrices are closed under $2$-sums we have that the matrix $N$, which
is the $2$-sum of the TU matrices $A$ and $B$, is TU. It remains to be shown
that changing a nonzero of $N$ to zero the matrix $N'$ so-obtained is also TU.
We consider the following two cases separately: (i) we replace a nonzero of
the submatrix $A'$ or $B'$ of $N$ by zero, and (ii) we replace a nonzero
element of the $ab$ submatrix of $N$ by zero. 

For case (i) we can assume without loss of generality that we change a nonzero element of $A'$ to zero and let us call $\bar{N}=\begin{bmatrix} \bar{A'} & ab\\0& 
B'\end{bmatrix}$ the matrix so-obtained (the case in which a nonzero element of $B'$ is changed is similar). Therefore, matrix $\bar{N}$ is the $2$-sum of the matrix $\begin{bmatrix} \bar{A'} & a\end{bmatrix}$ and $\begin{bmatrix}b\\B'\end{bmatrix}$, where $\begin{bmatrix} \bar{A'} & a\end{bmatrix}$ is a TU matrix since it is obtained from the SU matrix $A$ by replacement of a nonzero by a zero, and $\begin{bmatrix}b\\B'\end{bmatrix}$ is TU since it is equal to matrix $B$. From the fact that TU matrices are closed under $2$-sums the result follows.

For case (ii), let $N'$ be the matrix obtained from changing a nonzero of the
$ab$ part of $N$ to zero. We shall show that $N'$ is TU. Since SU matrices are closed under row and column permutations, we can assume that 
$N'=
\left[
\begin{array}{cc|c}
 A' &  a_1 &  ab_2  \\
0 &  b_1 &  B_1 
\end{array}
\right]
$, where $a_1$ contains the nonzero having changed and thus differs from
column $a$ only to that element, $b_1$ is the
first column of $B'$ and $B_1$ is the rest of it, i.e. $B'=[b_1 \; B_1]$, and
$B$ has as first row the vector $[1 \;  b_2]$, where the first element is $1$ since we assumed that $a$ is not a zero vector, i.e. 
$B=\left[\begin{array}{cc}
1 & b_2 \\
b_1 & B_1
\end{array}
\right]$. We can easily see that $N'$ is the $3$-sum of the following two matrices 
\[
\hat{A}=
\left[
\begin{array}{cc|cc}
A' & a_1 & a & a \\
0 & 1 & 0 & 1
\end{array}
\right]    
\
\textrm{   and   }
\hat{B}=
\left[
\begin{array}{ccc}
0 & 1 & b_2 \\
b_1 & b_1 & B_1
\end{array}
\right]
\]
Since TU matrices closed under $3$-sums it suffices to show that each of
$\hat{A}$ and $\hat{B}$ is TU. We know that $[A \; a \; a]$ is SU because of
Lemma~\ref{lem_opra}~(iv); moreover, from (iii) of the same Lemma we have that 
$
\left[
\begin{array}{ccc}
A' & a & a \\
0 & 0 & 1
\end{array}
\right]    
$
is SU. Applying again (iv) of Lemma~\ref{lem_opra}, we have that
$\tilde{A}=
\left[
\begin{array}{cc|cc}
A' & a & a & a \\
0 & 1 & 0 & 1
\end{array}
\right]    
$
is SU. Thus, changing a specific nonzero from a column 
$\left[
\begin{array}{r}
a \\
1  
\end{array}
\right] $ of $\tilde{A}$ to zero we obtain
$\hat{A}$ which has to be TU.
For $\hat{B}$ now, by the fact that $B$ is SU and Lemma~\ref{lem_opra}, the matrix $\tilde{B}=
\left[
\begin{array}{rrr}
1 &   1  & b_2 \\
b_1 &  b_1 & B_1
\end{array}
\right]
$ is SU. Thus, replacing a $1$ of a column
$\left[
\begin{array}{r}
1 \\
b  
\end{array}
\right] $ of $\tilde{B}$ we obtain matrix $\hat{B}$ which has to be TU. Since both
$\hat{A}$ and $\hat{B}$ are TU the result follows.
\end{proof}

In what follows we shall make use of the following regular matroid decomposition theorem by Seymour~\cite{Seymour:1980}.   
\begin{theorem} \label{th_Seym}
Every regular matroid $M$ may be constructed by means of $1$-, $2$-, and
$3$-sums starting with matroids each isomorphic to a minor of $M$ and each
either graphic or cographic or isomorphic to $R_{10}$.
\end{theorem}
\noindent
The $R_{10}$ regular matroid is a ten-element matroid, which can be found in~\cite{Oxley:06, Truemper:98}, and it has the following two unique totally unimodular compact representation matrices $B_1$ and $B_2$, up to row and column permutations and scaling of rows and columns by $-1$
\begin{equation}\label{eq_B1}
B_1=
\kbordermatrix{\mbox{}& 1 & 2 & 3 & 4 & 5 \\
1 & {\;\,\! 1}  & {\;\,\! 0}  & {\;\,\! 0}  & {\;\,\! 1}  & {\!\!\! -1}  \\  
2 & {\!\!\! -1} & {\;\,\! 1}  & {\;\,\! 0}  & {\;\,\! 0}  & {\;\,\! 1} \\  
3 & {\;\,\! 1}  & {\!\!\! -1} & {\;\,\! 1}  & {\;\,\! 0}  & {\;\,\! 0} \\  
4 & {\;\,\! 0}  & {\;\,\! 1}  & {\!\!\! -1} & {\;\,\! 1}  & {\;\,\! 0} \\  
5 & {\;\,\! 0}  & {\;\,\! 0}  & {\;\,\! 1}  & {\!\!\! -1} & {\;\,\! 1} 
}
\mspace{40mu}
B_2=\kbordermatrix{\mbox{}& 1 & 2 & 3 & 4 & 5 \\
1 & 1 & 1 & 1 & 1 & 1  \\ 
2 & 1 & 1 & 1 & 0 & 0 \\ 
3 & 1 & 0 & 1 & 1 & 0 \\  
4 & 1 & 0 & 0 & 1 & 1 \\  
5 & 1 & 1 & 0 & 0 & 1 \\ 
}
\end{equation}
A consequence of theorem~Theorem \ref{th_Seym}
is the construction Theorem~\ref{Seymour_matrix} for totally unimodular
matrices which appears in \cite{Seymour:95,Truemper:98}. 
\begin{theorem} \label{Seymour_matrix}
Any TU matrix is up to row and column permutations and scaling by $\pm{1}$ factors a network matrix, the transpose of a network matrix, the matrix $B_1$
 or $B_2$ of (\ref{eq_B1}),or may be constructed recursively by these matrices using  matrix $1$-, $2$- and $3$-sums. 
\end{theorem}
According to Theorem~\ref{Seymour_matrix}, the building blocks of totally
unimodular matrices are network matrices and their transposes as well as the
matrices $B_1$ and $B_2$ in \eqref{eq_B1}.
\begin{lemma}\label{lem_bb}
$B_1$ and $B_2$ are not SU.
\end{lemma}
\begin{proof}
If we make the value of the $(4,3)^{\textrm{th}}$-element of $B_1$ from $-1$ to $0$ then in the matrix so-obtained the
$3\times{3}$ submatrix defined by rows $3,4$ and $5$ and columns $2,3$ and $4$ 
has determinant equal to $+2$. Therefore, $B_1$ is not SU. Similarly, if we make the value of the $(4,1)^{\textrm{th}}$-element of
$B_2$ from $+1$ to $0$ then in the matrix so-obtained, the $3\times{3}$
submatrix defined by rows $3,4$ and $5$ and columns $1,4$ and $5$ has
determinant equal to $-2$ and thus, $B_2$ is not SU.  
\end{proof}
By Theorem~\ref{Seymour_matrix} and Lemma~\ref{lem_bb} we obtain the following result.
\begin{theorem}
Any SU matrix is up to row and column permutations and scaling by $\pm{1}$ factors a network 
matrix, the transpose of a network matrix, or may be constructed recursively by these matrices using  matrix $1$-, $2$- and $3$-sums. 
\end{theorem}
The following theorem, known as the splitter theorem for regular matroids, is one of the most important
steps which led to the regular matroid decomposition theorem~\cite{Seymour:1980}.
\begin{theorem}\label{th_r100}
Every regular matroid can be obtained from copies of $R_{10}$ and from
$3$-connected minors without $R_{10}$ minors by a sequence of $1$-sums and $2$-sums.
\end{theorem}
Combining the above we can now state the main result of this section. 
\begin{theorem} \label{th_lak}
A matrix is SU if and only if it is decomposable via $1$- and $2$-sums into strongly unimodular matrices 
representing $3$-connected regular matroids without $R_{10}$ minors.
\end{theorem}
\begin{proof}
The ``if part" follows directly from Lemmata~\ref{lem_1-s},~\ref{lem_2-s}. For the ``only if" part, let $A$ be an SU matrix. By definition, $A$ is TU and therefore, by Theorem~\ref{th_r100}, may be obtained from $1$- and $2$-sums from matrices representing $R_{10}$ and $3$-connected matroids without $R_{10}$ minors. By Lemma~\ref{lem_bb}, the two unique representations for  $R_{10}$ are not SU and therefore, $A$ can only be obtained from $3$-connected matrices without $R_{10}$ minor.
\end{proof}
In view of Theorem~\ref{th_lak} we can see that an SU matrix can be decomposed via $1$-sums and $2$-sums into a special class of SU matrices. 
This class will be characterised in the following section.
\section{The Network Structure of the Decomposition Blocks} \label{sec_3c}
By Theorem~\ref{th_lak} we have that SU matrices are decomposable into smaller SU matrices which represent $3$-connected regular matroids without 
$R_{10}$ minors. In this section we shall characterise the structure of these smaller matrices in Theorem~\ref{th_final}.

It is known that any $3$-connected binary matroid contains the wheel matroid $\mathcal{W}_3$ as a minor (Lemma~5.2.10 in~\cite{Truemper:98}), that is the graphic matroid with representation the wheel graph $W_3$ (i.e. the undirected graph obtained from the graphs in Figure~\ref{fig_w3} by omitting the directions).
In the following result we show that there exist two TU representation matrices for $\mathcal{W}_3$, one SU and one non-SU.
\begin{lemma} \label{lem_w3ne}
Up to row and column permutations and scaling by $-1$, the matroid $\mathcal{W}_3$ has two different totally unimodular compact representation matrices, namely
\begin{enumerate}
\item[(i)] an SU representation 
$
N_1=\left[
\begin{array}{rcc} 
 1 & 0 & 1 \\ 
-1 & 1 & 0 \\  
 0 & 1 & 1    
\end{array}
\right]
$, and 
\item[(ii)] a non-SU representation
 $
N_2=\left[
\begin{array}{ccc} 
 1 & 1 & 0 \\ 
 0 & 1 & 1 \\  
 1 & 1 & 1    
\end{array}
\right]
$       
\end{enumerate}  
\end{lemma}
\begin{proof}
Since the graphic matroids are uniquely representable over any field, given a TU  compact representation of $\mathcal{W}_3$
we can obtain any other compact representation by row and column permutations, scaling of rows and columns by $-1$ and pivoting.
Since $\mathcal{W}_3$ is a graphic matroid, each of its TU compact representation matrices is a network matrix as well.  
Pivoting in a network matrix results to a network matrix with respect to another spanning tree of the same graph. 
Specifically, up to graph isomorphism, graph $W_3$ has two different spanning trees which are depicted in Figure~\ref{fig_w3}, where solid edges 
correspond to the tree edges. Thus, up to row and column permutations and scaling by $-1$, there are two different network matrices 
representing $\mathcal{W}_3$; namely:
$
N_1=\left[
\begin{array}{rcc} 
 1 & 0 & 1 \\ 
-1 & 1 & 0 \\  
 0 & 1 & 1    
\end{array}
\right]
$, and 
$
N_2=\left[
\begin{array}{ccc} 
 1 & 1 & 0 \\ 
 0 & 1 & 1 \\  
 1 & 1 & 1    
\end{array}
\right]
$. It is now easy to see that if we replace any nonzero of     
$
\left[
\begin{array}{rcc} 
 1 & 0 & 1 \\ 
-1 & 1 & 0 \\  
 0 & 1 & 1    
\end{array}
\right]
$ by a $0$ then all the matrices so-obtained are TU. On the other hand, if we replace the nonzero at third row and  second column of 
 $
\left[
\begin{array}{ccc} 
 1 & 1 & 0 \\ 
 0 & 1 & 1 \\  
 1 & 1 & 1    
\end{array}
\right]
$ by $0$ then the matrix so-obtained is not TU.

\end{proof}
\begin{figure}[h] 
\begin{center}
\centering
\psfrag{(1)}{\footnotesize $(1)$}
\psfrag{(2)}{\footnotesize $(2)$}
\includegraphics*[scale=0.5]{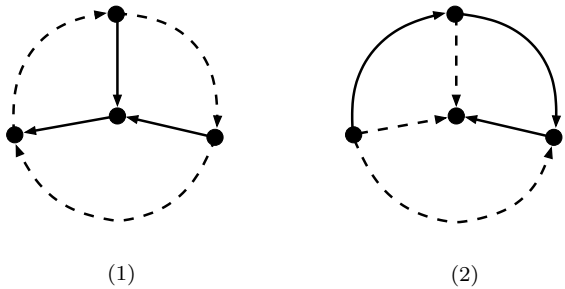}
\end{center}
\caption{The two possible network representations of graph $W_3$ where the network matrix associated with (1) 
is SU while in (2) it is non-SU.}
\label{fig_w3}
\end{figure}

We shall now prove the following important theorem which shows that SU representation matrices of $3$-connected regular matroids 
can not have certain $2\times{2}$ matrices as submatrices.
\begin{theorem} \label{th_tr23}
If $N$ is an $m\times{n}$ representation matrix $(m,n\geq{3})$ of a $3$-connected regular matroid
containing, up to row and column permutations and scalings by $-1$, the  submatrix
$
\left[
\begin{array}{rr} 
 1 &  1  \\
 1 &  1  
\end{array}
\right]
$, 
then $N$ is not SU.
\end{theorem}
\begin{proof}
Since $N$ is the representation matrix of a connected matroid we have that
it has an $M(W_2)$ minor (see Lemma~5.2.10 in \cite{Truemper:98}), where $W_2$ is the wheel graph with two spokes. Furthermore,
the matrix 
$
\left[
\begin{array}{rr} 
 1 &  1  \\
 1 &  1  
\end{array}
\right]
$ under any row and column permutations and scalings by $-1$
factors  displays $M(W_2)$. Enlarge this $2\times{2}$ submatrix
to a maximal submatrix containing only $1$s. Let us call $D$ that submatrix and index its
rows and columns by $R$ and $S$, respectively. Furthermore, in
the partitioned $N$, as it is depicted in \eqref{eq_1} below,  each row of the submatrix $U$ and each column of
the submatrix $V$ is assumed to be nonzero. From our assumption that $D$ is
maximal we have that each row and each column of $U$ and $V$, respectively, must have at least one
zero element.
\begin{equation} \label{eq_1}
N  = 
\kbordermatrix{ 
       &  S         &            & Q                     &              &                  \\
R     &  D        & \vrule  & V                     & \vrule   & 0                \\ \cline{2-6}
P     &  U        & \vrule  &  \{0,\pm{1}\}  & \vrule   & \{0,\pm{1}\}   \\ \cline{2-6}
       &  0         & \vrule  &  \{0,\pm{1}\}  & \vrule   & \{0,\pm{1}\}   
} 
\end{equation}    
Let $BG(N)$ be the bipartite graph of $N$ and let $F$ be its
subgraph obtained from the deletion of the edges corresponding to the $1$s of
$D$. Since $N$ is the representation matrix of
a $3$-connected regular matroid,  we have that there must exist a path in $F$
connecting a vertex of $R$ with a vertex of $S$ (see Lemma~5.2.11 in~\cite{Truemper:98}) which, due to the
bipartiteness of $F$, it has to be of odd
length. If we assume that the length of that path is $3$ then
the matrix $N_2$ of Lemma~\ref{lem_w3ne} is a submatrix of $N$, 
which implies that $N$ is not SU. 

If the shortest path connecting a vertex of $R$  with a vertex of $S$ has
length greater than $3$ then we
will show that the matrix $N$ is also non-SU. Let's say
that the shortest path lies between the vertices $r_2$ and $s_2$ of $R$ and
$S$, respectively. Then $N$ will have the following
submatrix $M$:
\begin{equation*} 
M=
\kbordermatrix{\mbox{}& q_1 & q_2 & & \ldots & & q_n & s_2 & s_1\\ 
r_2          & \pm{1} & 0           & 0 &             & 0 &  0  &  \pm{1} & \pm{1}\\   
p_n         & \pm{1} & \pm{1}  & 0 & \ldots  & 0 &  0  &   0& 0 \\ 
p_{n-1}    & 0           & \pm{1}  & \pm{1}  &   &0 &0 &0 &0\\  
\vdots    &  &   \vdots     &   & \ddots &   & &\vdots &  \\
p_1         & 0        &    0      & 0    &  & 0 & \pm{1} & \pm{1} & 0\\ 
r_1          & 0  & 0  &  0 &\ldots & 0 & 0 & \pm{1} & \pm{1} 
}
\end{equation*} 
where $\{r_1,r_2\}\in{R}$, $\{s_1,s_2\}\in{S}$, $\{p_1,\ldots,p_n\}\in{P}$ and $\{q_1,\ldots,q_n\}\in{Q}$.
Moreover, we  have that  $M$ will have  no zeros in the main diagonal and in the diagonal below the
main  because of the path existing between $r_2$ and $s_2$. The submatrix of $M$ having rows indexed by $r_1$ and $r_2$ and columns indexed by
$s_1$ and $s_2$ is full of ones because it is submatrix of $D$. Furthermore,
we have zeros in the position indexed by $r_1$ and $q_1$ and in the position
indexed by $p_1$ and $s_1$ because we can assume that there exists at least one
vertex of $R$ not being adjacent to $q_1$, which we call $r_1$, and similarly we
can assume that there exists a vertex of $S$ not being adjacent to $p_1$, which
we call $s_1$. All the other zeros in $M$ are due to the fact that the path
between $r_2$ and $s_2$ is the shortest between a vertex of $R$ and a vertex of
$S$ in the graph $F$.

We shall now show that matrix $M$ is not SU. If we
expand the determinant of $M$ along the first row then this determinant is  equal
to the sum of the determinants of three TU matrices being triangular with no
zero in the diagonal. Therefore, it is easy now to see that there exists a nonzero in the
first row of $M$ such that if we replace it by a zero and expand the determinant
of the matrix so-obtained along the first row then we have that the determinant of this
matrix will be $2$ or $-2$. Therefore, $N$ has a submatrix $M$ being non-SU
and by Lemma~\ref{lem_ew}, $N$ is not $SU$.                 
\end{proof}
Crama et al. in~\cite{CraLoPo:92}  proved that if $A$ is an SU matrix then we can partition its  rows as stated in the following theorem.
\begin{theorem} \label{th_cr11}
If $A$ is an SU matrix, then there exists a partition $(S_1,\ldots,S_k)$ of the rows of $A$ with the following properties:
\begin{itemize}
\item[(i)] every column of $A$ has $0, 1$ or $2$ nonzero entries in each $S_i$, for $i=1,\ldots,k$;
\item[(ii)] if a column has exactly one nonzero entry in some $S_i$, then all its entries in $S_{i+1},\ldots,S_k$ are zeros.
\end{itemize}
\end{theorem} 

Since by (i) of Lemma~\ref{lem_opra}, SU matrices are closed under taking the transpose we can restate Theorem~\ref{th_cr11}
for the columns of an SU matrix. Consider an SU matrix $A'$ and let
$\mathcal{S}=(S_1, S_2,\ldots,S_k)$ be the partition of its rows as
determined by Theorem~\ref{th_cr11} and
$\mathcal{T}=(T_1, T_2,\ldots,T_l)$ be the partition of the rows of the transpose of $A'$ 
as determined by Theorem~\ref{th_cr11}. Then by permuting rows and
columns of $A'$ we can obtain the following SU matrix $A$:  
\begin{equation} \label{eq_3}
A  = 
\kbordermatrix{ 
           &  T_1       &    T_2                &  \cdots    &  T_l              \\
S_1      &  A_{1,1}  & A_{1,2}            &   \cdots    & A_{1,l}            \\ 
S_2      &  A_{2,1}  & A_{2,2}             &   \cdots    & A_{2,l}            \\ 
\vdots & \vdots & \vdots            & \ddots     & \vdots         \\ 
S_k      & A_{k,1}  & A_{k,2}            &   \cdots    & A_{k,l}            \\ 
} 
\end{equation}    
where we have that each $A_{i,j}$ is the submatrix
of $A'$ defined by the rows of $S_i$ and columns of $T_j$.  We are now ready to state the main result of this section.

\begin{theorem}\label{th_final}
Let $A$ be an SU matrix representation of a $3$-connected regular  matroid  being in
the form of (\ref{eq_3}). Then the following hold:
\begin{itemize}
\item[(i)] $A_{1,1}$ has $0$ or $2$ non-zeros in each column and row 
\item[(ii)] each column of $A_{1,j}$ has $0$ or $2$ non-zeros and each row of
  $A_{i,1}$ has $0$ or $2$ nonzero elements
\item[(iii)] if an  $A_{i,j}$ has $2$ non-zeros in each column and each
 row then, up to row and column permutations,    
\[
A_{i,j}=
\left[
\begin{array}{cccccc} 
\pm{1}         &       &           &          &       & \pm{1}      \\
\pm{1}         & \pm{1}   &           &          &       &           \\ 
            & \pm{1}  & \pm{1}   &          &       &            \\  
            &       &           & \ddots &       &             \\  
            &       &           &          & \pm{1}   &            \\ 
            &       &           &          & \pm{1}   & \pm{1}             
\end{array}
\right]
\]
\end{itemize}
\end{theorem}
\begin{proof}
For (i) and (ii), by way of contradiction, it is enough to observe that if there was a column (row)
with exactly one nonzero, then by Theorem~\ref{th_cr11} this column (row) would be a unit column (row). This
would mean that the matroid represented by $A$ has a $2$-separation (see e.g. Lemma~3.3.20 in~\cite{Truemper:98}), which contradicts our hypothesis that this matroid is $3$-connected.


For (iii), from  Theorem~\ref{th_tr23} we have that $A_{i,j}$ can not have the
matrix 
$
\left[
\begin{array}{rr} 
 1 &  1  \\
 1 &  1  
\end{array}
\right]
$ as submatrix. It is now straightforward to see that $A_{i,j}$ has the form
described in (iii).
\end{proof}

\section{Conclusion}
A new decomposition theory for SU matrices with blocks being matrices representing simple networks has been proposed. Specifically, whenever an SU matrix is not network then it has a clear network structure since each block in the aforementioned decomposition of $A$ is the node-incident matrix of a directed graph. Moreover, decomposition matroidal results were related to SU matrices with the prospect to be utilized in real-life important problems. Such a field is that of complex networks modelling numerous real-life problems (see e.g.~\cite{EasKle:10,Jack:10,New:2010}). Most importantly, we strongly believe that this decomposition may be used for the development of a recognition algorithm for SU matrices which will not depend on the total-unimodularity test and would be a much more efficient, since it will utilize the network structure of the blocks of SU matrices.

\bibliographystyle{plain}

\end{document}